\documentclass[psamsfont,a4paper,12pt]{amsart}

\usepackage[english]{babel}   %loads english
\usepackage{young}

\usepackage{setspace}     %controls line-spacing
\usepackage{enumerate} 	%allows lists
\usepackage{afterpage} 	%keeps tables and figures where they belong
\usepackage[usenames,dvipsnames]{color}
\usepackage{graphicx}	%allows inclusion of pictures

\usepackage{amsmath} 	%just apparently the most important thing ever
\usepackage{amssymb}	%symbols in math
\usepackage{amsfonts} 	%standard fonts
\usepackage{latexsym} 	%more symbols
\usepackage[latin1]{inputenc} %sets up proper special characters

\usepackage{amsthm} 	%sets up theorem styling
\usepackage{stackrel}       %allows stacking in mathmode  
\usepackage{amscd} 	%commutative diagrams                                 

\usepackage{tabularx}	%automatically widening tables
\usepackage{longtable}	%for tables longer than a page

\usepackage[vcentermath]{youngtab}    %Young tableaux

\usepackage{verbatim} %comment environment?
\usepackage{blindtext} %something about dummy text

\usepackage{lipsum}

\allowdisplaybreaks

\renewcommand{\to}{\mapsto}

\newtheorem{Thm}{Theorem}[section]		%%with numbering

\newtheorem{Cor}[Thm]{Corollary}

\newtheorem*{thm*}{Theorem}	%%without numbering

\theoremstyle{definition}
\newtheorem{Definition}[Thm]{Definition}

\theoremstyle{remark}
 
\newtheorem*{rmk}{Remark}
\newtheorem{Example}[Thm]{Example}

\newtheorem{ind}[]{{\rm\it Indice}}

\usepackage{multicol}

\usepackage{tikz}
\usetikzlibrary{arrows}

\usepackage{paralist}
\usepackage{cite}

\usepackage{breqn}   %split parentheses across multiple lines

\author[Schneider]{Robert Schneider}
\address{Department of Mathematics\newline
University of Georgia\newline
Athens, Georgia 30602}
\email{robert.schneider@uga.edu}

\author[Sellers]{James A. Sellers}
\address{Department of Mathematics and Statistics\newline
University of Minnesota Duluth\newline
Duluth, Minnesota 55812}
\email{jsellers@d.umn.edu}

\author[Wagner]{Ian Wagner}
\address{Department of Mathematics\newline
Vanderbilt University\newline
Nashville, Tennessee 37240}
\email{ian.c.wagner@vanderbilt.edu}

\title[Sequentially congruent partitions and partitions into squares]{Sequentially congruent partitions\\ and partitions into squares}

\begin{document}
\numberwithin{equation}{section}

%\begin{abstract}

%\documentclass[12pt]{amsart}
%\usepackage{graphics,amsmath,amsthm,amssymb}
%\usepackage{color}
%%\usepackage{graphics,amsmath,amssymb,a4wide}
%\usepackage{mathtools} % for "\DeclarePairedDelimiter" macro
%\DeclarePairedDelimiter{\floor}{\lfloor}{\rfloor}
%\DeclarePairedDelimiter{\ceil}{\lceil}{\rceil}

\newtheorem{theorem}{Theorem}[section]
\newtheorem{proposition}[theorem]{Proposition}
\newtheorem{corollary}[theorem]{Corollary}
\newtheorem{definition}[theorem]{Definition}
\newtheorem{remark}[theorem]{Remark}
\newtheorem{example}[theorem]{Example}
\newtheorem{notation}[theorem]{Notation}

\def\qed{\hfill \rule{5pt}{7pt}}
\newcommand{\ndiv}{\hspace{-4pt}\not|\hspace{2pt}}

%
%
%\begin{document}
%
%\title[]{Sequentially Congruent Partitions and Partitions Into Square Parts}
%
%\author{Robert Schneider}
%%\address{Department of Mathematics\\
%%Harvey Mudd College\\
%%Claremont, CA 91711, USA\\
%%{\tt benjamin@hmc.edu}
%%}
%
%\author{James A. Sellers}
%%\address{Department of Mathematics, Penn State University\\ 
%%University Park, PA  16802, USA\\
%%{\tt sellersj@psu.edu}
%%}
%
%\author{Ian Wagner}
%
%
%
%\begin{abstract}

\begin{abstract}
In recent work, M. Schneider and the first author studied a curious class of integer partitions called ``sequentially congruent'' partitions: the $m$th part is congruent to the $(m+1)$th part modulo $m$, with the smallest part congruent to zero modulo the number of parts. 
Let $p_{\mathcal S}(n)$ be the number of sequentially congruent partitions of $n,$ and let $p_{\square}(n)$ be the number of partitions of $n$ wherein all parts are squares.  In this note we prove bijectively, for all $n\geq 1,$ that $p_{\mathcal S}(n) = p_{\square}(n).$ Our proof naturally extends to show other exotic classes of partitions of $n$ are in bijection with certain partitions of $n$ into $k$th powers. 

\end{abstract}

%\date{\today}

\maketitle
% \bigskip

%\noindent 2010 Mathematics Subject Classification: 05A17
%\bigskip

%\noindent Keywords: partition, sequentially congruent, conjugate

 \section{Euler, Cauchy, and sequentially congruent partitions}
 
Let $\mathcal P$ denote the set of integer partitions%, i.e., all finite multisets of natural numbers
\footnote{See \cite{Andrews_theory} for further reading about integer partitions.}. Euler connected partitions to analysis % --- and created the template for how we view partition theory --- 
with his generating function for the {partition function} $p(n)$, which gives the number of partitions of $n \geq 0$:
\begin{equation}\label{Euler}
\prod_{k=1}^{\infty}\frac{1}{1-q^k}=\prod_{k=1}^{\infty}\left(\sum_{j=0}^{\infty} q^{jk} \right) % =\prod_{k=1}^{\infty}(1+q^k+q^{k+k}+q^{k+k+k}+...) 
=\sum_{n=0}^{\infty}p(n)q^n,
\end{equation}
where $q\in \mathbb C, |q|<1,$ and we define $p(0):=1$. %We write the middle product above to suggest how each geometric series factor contributes some number of identical parts, to the partitions enumerated on the right.
Another important classical result is the Cauchy product formula giving the product of two convergent power series:
\begin{equation}\label{Cauchy}
\left(\sum_{i=0}^{\infty} a_i q^i \right)\left(\sum_{j=0}^{\infty} b_j q^j \right)=\sum_{n=0}^{\infty}q^n\sum_{k=0}^{n}a_k b_{n-k}.
\end{equation}
%where again $|q|<1$. 
This result extends to the product
$\prod_{k}\left(\sum_{j=0}^{\infty} a_{k,j} q^j \right)
$ of any number of power series (even infinitely many), with more complicated multiple sums comprising the resulting coefficients.

Now we consider what happens if we interpret the middle product of \eqref{Euler} in terms of \eqref{Cauchy}, that is, if we expand the product of geometric series via the Cauchy product formula, instead of collecting coefficients to count partitions of $n$ as in Euler's treatment. Since each geometric series $\sum_j q^{jk}=\sum_i a_{i,k} q^i$ has coefficient $a_{i,k}=1$ if $k$ divides $i$ and $a_{i,k}=0$ otherwise, this calculation yields 
\begin{flalign}\label{Cauchy2}
\prod_{k=1}^{\infty}\left(\sum_{j=0}^{\infty} q^{jk} \right) =\sum_{n=0}^{\infty}q^n \sum_{k_2=1}^{n}\sum_{k_3=1}^{k_2}\sum_{k_4=1}^{k_3}...\sum_{k_{n}=1}^{k_{n-1}}\{\text{either $0$ or $1$\}},
\end{flalign}
where the coefficient of $q^n$ on the right-hand side is an $(n-1)$-tuple sum whose summands enumerate some class of partitions induced by particular combinations of the indices of the multiple sum\footnote{See \cite{Schneider_zeta}, Cor. 2.9, for explicit treatment of a more general case of \eqref{Cauchy2}.}. 
But what {\it are} these seemingly-complicated partitions?
 
% \section{Sequentially congruent partitions}
In fact, the right-hand side of \eqref{Cauchy2} enumerates certain {\it sequentially congruent partitions}%\footnote{In this case, sequentially congruent partitions of length $n$}
, a subset of $\mathcal P$ studied by M. Schneider and the first author in \cite{SS}.
 
 \begin{Definition} \label{SCdef}
A partition $\lambda = ( \lambda_{1}, \lambda_{2}, \dots, \lambda_{r})$, $\lambda_{1} \geq \lambda_{2} \geq \cdots \geq \lambda_{r} \geq 1$, is \textit{sequentially congruent} if: 
\begin{enumerate}
\item $\lambda_{i} \equiv \lambda_{i+1} \pmod{i}$ for $1 \leq i \leq r-1$; and
\item $\lambda_{r} \equiv 0 \pmod{r}$.
\end{enumerate}
We let $\mathcal{S} \subset \mathcal{P}$ denote the set of sequentially congruent partitions.
\end{Definition}

\begin{example}
The partition $(20,17,15,9,5)$ 
is sequentially congruent, because $20 \equiv 17\  (\operatorname{mod}\  1)$ trivially, $17\equiv 15\  (\operatorname{mod}\  2)$, $15\equiv 9\  (\operatorname{mod}\  3)$, $9\equiv 5\  (\operatorname{mod}\ 4)$, and finally $5\equiv 0\  (\operatorname{mod}\  5)$. 
\end{example}

%Modulo its index, each part of a sequentially congruent partition is congruent to the subsequent part, and the smallest part is divisible by the length $r$. 
With such strict congruence conditions imposed on the parts, sequentially congruent partitions seem to comprise a somewhat artificial subset of $\mathcal P$. However, in \cite{SS} %it is proved that the number of sequentially congruent partitions with largest part $n$ is equal to $p(n)$, the number of partitions of size $n$: the two sets are in bijection.
they are found to fit nicely into %classical 
partition theory, %\footnote{And are also shown connect with G. E. Andrews' theory of partition ideals (see \cite{Andrews_theory}).}, 
as the following theorem exemplifies. %, and in this paper we prove they are in bijection with a familiar classical subset of $\mathcal P$.

\begin{thm*}[Schneider--Schneider]
The set $\mathcal S$ enjoys a natural bijection with the set $\mathcal P$. Moreover, the number %$$ \textcolor{gray}{\blacksquare}$ \mathcal S_{\text{lg}=n}$ 
of sequentially congruent partitions with largest part $n$ is equal to $p(n)$.
\end{thm*}

This bijection can be inferred by comparing coefficients on the right-hand sides of \eqref{Euler} and \eqref{Cauchy2}. In \cite{SS}, {two} different bijective maps between $\mathcal S$ and $\mathcal P$ are given, producing two combinatorial proofs of the theorem. Surprisingly, composing either map with the inverse of the other map produces partition {\it conjugation}, i.e.,  the interchange of rows and columns in the Young diagram of $\lambda$. Of course, conjugation preserves partition {\it size} $|\lambda|:=\lambda_1+\lambda_2+\cdots +\lambda_r$, % (sum of parts), 
thus the aforementioned composite map takes the set of partitions 
%$\mathcal P_n$ 
of size $n$ to itself. 

Presented with these size-$n$ correspondences, an obvious question to ask now is: what is the number $p_{\mathcal S}(n)$ of sequentially congruent partitions of size $n$? Interestingly, a different bijection involving $\mathcal S$ gives a very satisfying answer.
%\begin{theorem}
%frequency congruent bijection
%\end{theorem}
%

\begin{Thm}\label{Thm1}
The set $\mathcal S$ enjoys a natural bijection with the set $\mathcal P_{
\square}$ of partitions whose parts are all perfect squares. Moreover, let  $p_{\square}(n)$ be the number of partitions of $n$ into squares. Then for all $n\geq 1,$ we have $$p_{\mathcal S}(n) = p_{\square}(n).$$ 
\end{Thm}

Historically, sums of squares are of deep interest in arithmetic. The clay tablet known as Plimpton 322 contains lists of Pythagorean triples dating back to ancient Babylon \cite{Robson}, and there are many examples in the partitions literature of connections to sums of squares (see e.g. \cite{Bell, HS, HS2, Wright}). It is interesting to see the seemingly-artificial set $\mathcal S$ fitting nicely, once again, into the theory of numbers\footnote{We note then that the case $k=2, M=1$ of Wright \cite{Wright}, Thm. 2, gives the asymptotic estimate $p_{\mathcal S}(n) \sim \frac{B e^{C{n^{1/3}}}}{n^{7/6}}$ with constants $B:=\left(\frac{\zeta(3/2)^{4}}{442368\pi^{7}}\right)^{1/6}, C:=\frac{3}{2}\left(\frac{{\pi}}{2}  \zeta(3/2)^2\right)^{1/3}$, where $\zeta$ is the Riemann zeta function.}. 

Combining Theorem \ref{Thm1} with the aforementioned theorem from \cite{SS} allows one to define a set of partitions into squares which is equinumerous with the number of partitions of $n$.  Let
\begin{equation*}
\lambda = \left( (1^2)^{e_{1}} \ (2^2)^{e_{2}} \ (3^2)^{e_{3}} \dots (i^2)^{e_{i}} \dots \right) \in \mathcal{P}_{\square}
\end{equation*}
be a partition into squares given in frequency notation (as defined in \eqref{freq} below), and let us define the following ``weighted frequency'' statistic $\mathcal L$ on partitions into squares:
\begin{equation} \label{L}
\mathcal{L} (\lambda) = \sum_{i \geq 1} i \cdot e_{i}.
\end{equation}
%Then we have the following corollary.
\begin{Cor} \label{Cor}
The set $\mathcal{P}_{\square}$ enjoys a natural bijection with the set $\mathcal{P}$.  Moreover, the number of partitions $\lambda$ whose parts are all perfect squares with $\mathcal{L}(\lambda)=n$ is equal to $p(n)$.
\end{Cor}

%
%Asymptotics for partitions into $k$th powers in general were studied in depth by E. M. Wright \cite{Wright}, following up on the well-known asymptotic formula of Hardy--Ramanujan: % for the partition function:
%\begin{equation}
%p(n) \sim \frac{e^{\pi\sqrt{2n/3}}}{4n \sqrt{3}},
%\end{equation}
%which represents the $k=1$ case. Using Wright's work, from Theorem \ref{Thm1} we can deduce an asymptotic for the number of sequentially congruent partitions of $n$. 
%\begin{Cor}\label{Cor1}
%As $n\to \infty$ we have the asymptotic estimate:
%     $$p_{\mathcal S}(n) \sim \frac{B e^{C{n^{1/3}}}}{n^{7/6}}$$
%with constants
%$$B:=\left[\frac{\zeta(3/2)^{4}}{442368\pi^{7}}\right]^{1/6},\  \  \   C:=\frac{3}{2}\left[\frac{{\pi}}{2}  \zeta(3/2)^2\right]^{1/3},$$
%where $\zeta$ denotes the Riemann zeta function, $\zeta(3/2)= 2.612375...$.
%\end{Cor}
%%We prove these formulas in Section \ref{Proofs} below.
%%
%%
%%Let $p_{\mathcal F}(n)$ denote the number of frequency congruent partitions of size $n$.
%%
%%\begin{Lemma}
%%For all $n\geq 1,$ we have $$p_{\mathcal F}(n) = p_{\square}(n).$$ 
%%\end{Lemma}
%%
%%
%%\begin{Cor}
%%Theorems hold for frequency congruent partitions
%%\end{Cor}
%%
%%
%
 
\section{Proofs, example, and extension of our result}
%\label{intro}

%\begin{proof}
We present a bijective proof of Theorem \ref{Thm1}.
Rather than proving the theorem directly from Young diagrams or other combinatorial devices (which seems possible), 
we proceed %It is conceivable (and an interesting combinatorial problem) to  prove the theorem directly from Young diagrams or other combinatorial devices. Perhaps a more elegant combinatorial proof can be attained, though, 
by considering the set of {\it conjugates} of the partitions in $\mathcal S$, that is, partitions formed by conjugation of the Young diagrams of sequentially congruent partitions. These conjugate partitions are also studied in \cite{SS}, where they are called {\it frequency congruent partitions}.% We willprove frequency congruent partitions of $n$ are in bijection with partitions of $n$ into squares. 

Let $m_i=m_i(\lambda)\geq 0$ denote the {\it frequency} (or ``multiplicity'') of $i\in \mathbb N$, the number of repetitions of $i$ as a part of partition $\lambda$. 

\begin{Definition}
A partition $\lambda$ is {\it frequency congruent} if for all $i\geq 1$, %The conjugate $\lambda^*$ of a partition $\lambda \in \mathcal S$ satisfies the 
%following condition: 
$$m_i(\lambda)\equiv 0\  (\text{mod}\  i).$$
We let $\mathcal{F} \subset \mathcal{P}$ denote the set of frequency congruent partitions.
\end{Definition}

We also recall the alternative ``frequency notation'' representation of a partition:
\begin{equation} \label{freq}
\lambda =(1^{m_1}\  2^{m_2}\   3^{m_3} \dots \ i^{m_i}\dots),\end{equation} 
where only finitely many of the frequencies $m_i$ are nonzero.

\begin{proof}[Proof of Theorem \ref{Thm1} and Corollary \ref{Cor}]
We begin by observing that the sets $\mathcal S$ and $\mathcal F$ are in bijection through conjugation, which preserves size and interchanges the largest part and number of parts.  Therefore, if $p_{\mathcal F}(n)$ denotes the number of frequency congruent partitions of $n$, we have
\begin{equation}\label{Fbij}
p_{\mathcal S}(n)=p_{\mathcal F}(n),
\end{equation}
and if $\lambda \in \mathcal{S}$ has largest part equal to $k$, then the conjugate $\lambda^* \in \mathcal{F}$ has $k$ total parts.

Let $\lambda$ be a partition of size $n$ which is sequentially congruent.  Then the conjugate $\lambda^*$ of $\lambda$ is {frequency congruent}, thus each part divides its own frequency: % partition (see \cite{SS}) of size $n.$  Namely, using frequency notation, 
%$$
%\lambda' = (1^{f_1}, 2^{f_2}, \dots, i^{f_i}, \dots)
%$$
%where $i$ divides $f_i$ for all $i.$  This means 
\begin{equation}
\lambda^* = (1^{1\cdot e_1}\  2^{2\cdot e_2}  \dots\  i^{i\cdot e_i}\ \dots)
\end{equation}
for some $e_i\geq 0$.  Now, let us define a map $\pi:\mathcal F \to \mathcal P_{\square}$ taking $\lambda^*\in \mathcal F$ to the partition 
\begin{equation}\label{pi}
\pi(\lambda^*)=((1^2)^{e_1}\  (2^2)^{e_2}  \dots\  (i^2)^{e_i}  \dots)\in \mathcal P_{\square},
\end{equation}
noting that $|\lambda^*|=|\pi(\lambda^*)|=n$ since for each $i\geq 1$, the sum of corresponding parts is preserved:
\begin{equation}\label{map}
\underbrace{i+i+\cdots+i}_{i\cdot e_i \ times}=i\cdot (\underbrace{i+i+\cdots+i}_{e_i \ times})= \underbrace{i^2+i^2+\cdots+i^2}_{e_i \ times}.\end{equation}
%Then $\sigma$ is a partition of size $n$ into square parts.  
The map is clearly reversible, so $\pi$ is a bijection. Therefore, we have
\begin{equation}\label{Fbij2}
p_{\mathcal F}(n)=p_{\square}(n).
\end{equation}
Comparing \eqref{Fbij} and \eqref{Fbij2} completes the proof of Theorem \ref{Thm1}.  

The proof of Corollary \ref{Cor} follows by noting that $\mathcal{L}(\pi(\lambda^*))$ is the number of parts (or ``length'') of $\lambda^*$.  The statistic $\mathcal{L}(\pi(\lambda^*))$ can be interpreted as a weighted sum of the number of squares of each size that appear in the Young diagram of the partition $\lambda^*$, as illustrated in Example \ref{ex2} below.
\end{proof}

 \begin{rmk}
 One could prove Theorem \ref{Thm1} analytically using the generating function proof of Corollary 4.3 in \cite{SS}, noting the left-hand product can be interpreted as generating either frequency congruent partitions or partitions into squares.
   \end{rmk}

%
%
%\begin{proof}[Proof of Corollary \ref{Cor1}]
%Let $p_k(n)$ denote the number of partitions of $n$ into $k$th powers; thus $p_{\square}(n)=p_2(n)$. In \cite{Wright}, Wright credits Hardy and Ramanujan for conjecturing an asymptotic
%\begin{equation}\label{asymp}
%p_{k}(n) \sim \frac{B_k e^{C_k{n^{\frac{1}{k+1}}}}}{n^{\frac{3k+1}{2(k+1)}}},
%\end{equation}
%with constants $B_k, C_k$ depending on $k$ in a fairly complicated way. Wright proves \eqref{asymp} as the first approximation of more precise estimates; Corollary \ref{Cor1} is equivalent to the $k=2, M=1$ case of Theorem 2 in \cite{Wright}, after the appropriate substitutions and simplifications.% (noting $B=B_2, C=C_2$ in \eqref{asymp}).
%\end{proof}
%
%
%

%
%%\newpage
%
%\section{Examples of the bijective map}

Here we give a concrete example of the bijection proved above.

%
%\begin{Example}
%Let $n=10.$  In the work below, the partitions on the left are the sequentially congruent partitions of size 10, the partitions in the middle are the corresponding frequency congruent partitions of 10 (obtained by conjugation of the sequentially congruent partitions), and the partitions on the right are the corresponding partitions into square parts:
%\begin{flalign}
%(10) & \mapsto (1^{10})\mapsto (1^{10})\\
%\nonumber (8,2) &\mapsto (1^6\  2^2)\mapsto (1^6\  4^1)\\
%\nonumber (6,4) &\mapsto (1^2\  2^4)\mapsto (1^2\  4^2)\\ 
%\nonumber (4,3,3)& \mapsto (1^1\  3^3)\mapsto (1^1\  9^1). \end{flalign}
%\end{Example}
%
\begin{Example} \label{ex2}
Begin with the frequency congruent partition 
\begin{equation}
\phi = (1^7\   2^6\  4^8\  5^5)\in \mathcal F.\end{equation}
%(where, as above, we have used frequency notation to write this partition).  
By conjugation, $\phi$ is mapped to the sequentially congruent partition
\begin{equation}(26, 19, 13, 13, 5)\in \mathcal S.\end{equation}
Moreover, using the ``squaring'' map $\pi$ defined in \eqref{pi}, $\phi$ is mapped to the partition 
\begin{equation}(1^7\  4^3\  16^2\  25^1)\in \mathcal P_{\square}.\end{equation} 
When we compose these two maps (and ignore the intermediary frequency congruent partition $\phi$ with which we began), we see the bijection
 \begin{equation}(26, 19, 13, 13, 5)\longleftrightarrow (1^7\  4^3\  16^2\  25^1). \end{equation}
This bijection is visually evident in the Young diagram of the sequentially congruent partition, which we shade to highlight that it breaks down into a concatenation of perfect squares:

\begin{flalign}
\begin{Young}
&     & &  & &   $ \textcolor{gray}{\blacksquare}$    & $ \textcolor{gray}{\blacksquare}$ &  $ \textcolor{gray}{\blacksquare}$  &$ \textcolor{gray}{\blacksquare}$  &      & &  &&  $ \textcolor{gray}{\blacksquare}$    &$ \textcolor{gray}{\blacksquare}$  &   &  &    $ \textcolor{gray}{\blacksquare}$    &  $ \textcolor{gray}{\blacksquare}$ &    &$ \textcolor{gray}{\blacksquare}$  &     &$ \textcolor{gray}{\blacksquare}$  &  & $ \textcolor{gray}{\blacksquare}$  & \cr
&     & &  & &     $ \textcolor{gray}{\blacksquare}$    & $ \textcolor{gray}{\blacksquare}$ &  $ \textcolor{gray}{\blacksquare}$  &$ \textcolor{gray}{\blacksquare}$  &       & &  &&   $ \textcolor{gray}{\blacksquare}$   &$ \textcolor{gray}{\blacksquare}$  &   &  &    $ \textcolor{gray}{\blacksquare}$    & $ \textcolor{gray}{\blacksquare}$  \cr
&     & &  & &      $ \textcolor{gray}{\blacksquare}$    & $ \textcolor{gray}{\blacksquare}$ &  $ \textcolor{gray}{\blacksquare}$  &$ \textcolor{gray}{\blacksquare}$  &         & & & \cr
&     & &  & &     $ \textcolor{gray}{\blacksquare}$    & $ \textcolor{gray}{\blacksquare}$ &  $ \textcolor{gray}{\blacksquare}$  &$ \textcolor{gray}{\blacksquare}$  &     & & & \cr
&  &  &  & \cr
\end{Young}
\end{flalign}
\begin{remark}Conversely, any collection of perfect squares may be concatenated in similar fashion (top edges aligned, weakly decreasing areas from left to right) to produce the Young diagram of a sequentially congruent partition. \end{remark} 
\end{Example}

%\section{Variation on Theorem \ref{Thm1}}
We now define a certain refinement of sequentially congruent partitions amenable to a combinatorial proof like the proof of Theorem \ref{Thm1} above, to show how similar methods yield other bijections of a similar flavor. 
%, noting that many such refinements are possible from the ideas in \cite{SS} as well as in the present work.

\begin{Definition}
Let $\mathcal{S}(j,k)\subset \mathcal P$ be the subset of partitions $\lambda = ( \lambda_{1}, \lambda_{2}, \dots, \lambda_{r})$ satisfying:
\begin{enumerate}
\item $\lambda_{i} - \lambda_{i+1} = j \cdot i^{k}$ for $1 \leq i \leq r-1$; and
\item $\lambda_{r} = j \cdot r^{k}$.
\end{enumerate}
\end{Definition}

\begin{Thm} \label{2}
The partitions $\lambda \in \mathcal{S}(j,k)$ of size $n$ are in bijection with partitions of $n$ into $(k +1)$th powers where each part occurs exactly $j$ times.
\end{Thm}

\begin{proof}[Proof of Theorem \ref{2}]
A very similar map to $\pi$, which we will denote by $\pi_{j,k}$, can be employed to give this result.  Notice that the conjugate $\lambda^*$ of a partition $\lambda\in \mathcal{S}(j,k)$ is a partition where $m_{i}(\lambda^{*}) = j \cdot i^{k}\epsilon_i$, with $\epsilon_i=\epsilon_i(\lambda^*)=1$ if $i$ is a part of $\lambda^*$ and $\epsilon_i=0$ otherwise.  Then we define
\begin{align}
\pi_{j,k}(\lambda^{*}) &= \pi_{j,k}((1^{j\cdot 1^k\epsilon_1}\  2^{j \cdot 2^k\epsilon_2}\  3^{j \cdot 3^k\epsilon_3}  \dots\  i^{j\cdot i^k\epsilon_i} \dots))\\& =\nonumber ((1^{k+1})^{j\epsilon_1}\  (2^{k+1})^{j\epsilon_2}\  (3^{k+1})^{j\epsilon_3}  \dots\  (i^{k+1})^{j\epsilon_i} \dots).
\end{align}
Noting that the map $\pi_{j,k}$ is both size-preserving and reversible, much like the map $\pi$, completes the proof of the bijection.
\end{proof}

It seems worthy of further investigation, to apply partition-theoretic maps like those presented here to study sums of squares and higher powers. %, between partitions and sums of squares and higher powers. 

\subsection*{Acknowledgments} The authors are grateful to Maxwell Schneider for conversations that informed our work, and to the anonymous referee for suggestions that strengthened this paper.

\end{document}